\documentclass[a4paper,	12pt]{article}

\usepackage[utf8]{inputenc}
\usepackage[T1]{fontenc}
\usepackage[english]{babel}

\usepackage[a4paper, margin=1in]{geometry}
\usepackage{microtype}	

\usepackage[autostyle=true]{csquotes} 
\usepackage[
	backend=biber,		
	style=numeric,		
	citestyle=numeric,	
	sorting=none,		
	giveninits=true,	
	doi=false,			
	url=false,			
	eprint=true,		
	isbn=false			
]{biblatex}

\usepackage{amsmath, amsthm, mathtools}	

\numberwithin{equation}{section} 

\theoremstyle{plain} 
\newtheorem{theorem}[equation]{Theorem}
\newtheorem{corollary}[equation]{Corollary}
\newtheorem{lemma}[equation]{Lemma}
\newtheorem*{lemma*}{Lemma}
\newtheorem{proposition}[equation]{Proposition}

\theoremstyle{definition} 
\newtheorem{definition}[equation]{Definition}
\newtheorem{example}[equation]{Example}

\usepackage{amssymb}

\newcommand{\RR}{\mathbb{R}}

\newcommand{\from}{\colon} 
\newcommand*{\at}[2]{{#1}|_{#2}}
\newcommand*{\abs}[1]{\left|{#1}\right|}
\newcommand*{\norm}[1]{\lVert{#1}\rVert}
\providecommand*{\de}{\mathop{}\!\mathrm{d}} 

\DeclareMathOperator{\diam}{diam}
\DeclareMathOperator{\dist}{dist}

\usepackage{newtxtext}
\usepackage{newtxmath}

\usepackage{comment}

\usepackage[hidelinks]{hyperref}

\usepackage{cleveref}
\crefformat{enumi}{(#2#1#3)}			

\title{Null Distance and Temporal Functions}
\author{Andrea Nigri\thanks{This work is based on the author's Master's thesis in Mathematics at the University of Milano-Bicocca.}}
\date{}

\begin{document}
	
\maketitle

\begin{abstract}
	The notion of null distance was introduced by Sormani and Vega in~\cite{SormaniVega2016} as part of a broader program to develop a theory of metric convergence adapted to Lorentzian geometry. Given a time function $\tau$ on a spacetime $(M,g)$, the associated null distance $\hat{d}_\tau$ is constructed from and closely related to the causal structure of $M$. While generally only a semi-metric, $\hat{d}_\tau$ becomes a metric when $\tau$ satisfies the local anti-Lipschitz condition.
	
	In this work, we focus on temporal functions, that is, differentiable functions whose gradient is everywhere past-directed timelike. Sormani and Vega showed that the class of $C^1$ temporal functions coincides with that of $C^1$ locally anti-Lipschitz time functions. When a temporal function $f$ is smooth, its level sets $M_t = f^{-1}(t)$ are spacelike hypersurfaces and thus Riemannian manifolds endowed with the induced metric $h_t$. Our main result establishes that, on any level set $M_t$ where the gradient $\nabla f$ has constant norm, the null distance $\hat{d}_f$ is bounded above by a constant multiple of the Riemannian distance $d_{h_t}$.
	
	Applying this result to a smooth regular cosmological time function $\tau_g$---as introduced by Andersson, Galloway, and Howard in~\cite{AnderssonGallowayHoward1998}---we prove a theorem confirming a conjecture of Sakovich and Sormani~\cite{SakovichSormani2025}: if the diameters of the level sets $M_t = \tau_g^{-1}(t)$ shrink to zero as $t \to 0$, then the spacetime exhibits a Big Bang singularity, as defined in their work.
\end{abstract}

\clearpage
\tableofcontents

\clearpage
\section{Introduction}

Let $(M,g)$ be a spacetime, that is, a connected, time-oriented Lorentzian manifold. A continuous function $\tau \from M \to \RR$ is called a \emph{time function} if it is strictly increasing along all future-directed causal curves. The null distance associated with a time function was defined by Sormani and Vega as follows.\footnote{In \cite{SormaniVega2016}, the null distance is defined with respect to a \emph{generalized time function}, which is not necessarily continuous. In this work, we restrict our attention to continuous time functions.}

\begin{definition}
	A curve $\beta \from [a,b] \to M$ is called a \emph{piecewise causal curve} if there exists a partition $a = s_0 < \dots < s_N = b$ such that each restriction $\beta_i = \at{\beta}{[s_i,s_{i+1}]}$ is either a future-directed or a past-directed causal curve. Given a time function $\tau$ on $M$, the \emph{null length} of a piecewise causal curve $\beta$ is defined by
	\[
		\hat{L}_{\tau}(\beta) = \sum_{i=0}^{N-1} \abs{ \tau(\beta(s_{i+1})) - \tau(\beta(s_i)) }.
	\]
	The \emph{null distance} between two points $p,q \in M$ is then given by
	\[
		\hat{d}_{\tau}(p,q) = \inf\{ \hat{L}_{\tau}(\beta) \mid \beta \text{ is a piecewise causal curve from $p$ to $q$} \}.
	\]
	If $q = p$, we set $\hat{d}_{\tau}(p,q) = 0$.
\end{definition}

It follows readily that $\hat{d}_{\tau}$ is a semi-metric on $M$. However, in general, it fails to be a metric, as it may vanish for distinct points. Sormani and Vega identified a condition on the time function---the \emph{local anti-Lipschitz condition}---which is equivalent to the definiteness of $\hat{d}_{\tau}$. When $\tau$ satisfies this condition, the associated null distance $\hat{d}_{\tau}$ is a metric on $M$ that induces the manifold topology.

A notable example of a time function is a \emph{temporal function}, that is, a differentiable function $f \from M \to \RR$ whose gradient $\nabla f$ is everywhere past-directed timelike.\footnote{We slightly deviate from the standard definition of a temporal function, which requires smoothness.} Sormani and Vega proved that the class of $C^1$ temporal functions coincides with that of $C^1$ locally anti-Lipschitz time functions. Consequently, if $f$ is a $C^1$ temporal function, then the associated null distance $\hat{d}_f$ is a metric on $M$ that induces the manifold topology. Furthermore, if $f$ is smooth and the level set $M_t = f^{-1}(t)$ is nonempty for some $t \in \RR$, then $M_t$ is a spacelike hypersurface. Hence, the restriction of the ambient Lorentzian metric $g$ to $M_t$, denoted by $h_t$, is a Riemannian metric on $M_t$. The following theorem represents our main result.

\begin{theorem}
	\label{thm:inequality}
	Let $(M,g)$ be a spacetime admitting a smooth temporal function $f \from M \to \RR$. Suppose there exist $t \in \RR$ and $C > 0$ such that the level set $M_t = f^{-1}(t)$ is nonempty and the equality
	\[
		g(\nabla f, \nabla f) = -C^2
	\]
	holds at every point of $M_t$. Then, for every $p,q \in M_t$, 
	\[
		\hat{d}_f(p,q) \leq Cd_{h_t}(p,q),
	\]
	where $h_t$ denotes the induced Riemannian metric on $M_t$ and $d_{h_t}$ is the corresponding distance function.
\end{theorem}

An example of a smooth temporal function with unit gradient everywhere is a smooth \emph{regular cosmological time function} $\tau_g \from M \to (0,\infty)$, as introduced by Andersson, Galloway, and Howard in \cite{AnderssonGallowayHoward1998}. Like any locally anti-Lipschitz time function, $\tau_g$ is $1$-Lipschitz with respect to the associated null distance $\hat{d}_{\tau_g}$, and hence admits a unique $1$-Lipschitz extension $\overline{\tau_g} \from \overline{M} \to [0,\infty)$ to the metric completion of $(M, \hat{d}_{\tau_g})$. This extension allows the definition of the initial level set $\overline{\tau_g}^{-1}(0) \subseteq \overline{M}$. 

In \cite{SakovichSormani2025}, Sakovich and Sormani considered the case in which the initial level set consists of a single point, in connection with their proposal for a general class of Big Bang spacetimes. Building upon Theorem~\ref{thm:inequality}, we establish the following result, which confirms Conjecture 3.7 of~\cite{SakovichSormani2025}.

\begin{theorem}
	\label{thm:initial_level_set}
	Let $(M,g)$ be a spacetime admitting a smooth regular cosmological time function $\tau_g \from M \to (0,\infty)$. For each $t > 0$, if the level set $M_t = \tau_g^{-1}(t)$ is nonempty, denote by $h_t$ the induced Riemannian metric on $M_t$ and by $d_{h_t}$ the corresponding distance function. If
	\[
		\lim_{t \to 0^+} \diam_{d_{h_t}}(M_t) = 0,
	\]
	then the initial level set $\overline{\tau_g}^{-1}(0) \subseteq \overline{M}$ consists of a single point $p_{BB} \in \overline{M}$.
\end{theorem}

\paragraph*{Outline.}

Section 2 introduces definitions and results from Lorentzian and metric geometry that will be used throughout. Section 3 focuses on the notion of null distance. Section 4 presents the proof of Theorem~\ref{thm:inequality}. Finally, Section 5 discusses regular cosmological time functions and establishes Theorem~\ref{thm:initial_level_set}.

\paragraph*{Acknowledgments.}

Part of this work was carried out during a research stay at the University of Vienna in the group of Prof.~Michael Kunzinger, which is partly supported by the Austrian Science Fund (FWF) [Grant DOI 10.55776/EFP6]. I am grateful to Chiara Rigoni for giving me the opportunity to join the group and for proposing the topic. I am especially thankful to Tobias Beran, whose insights and ideas have been a valuable source of inspiration. I also wish to thank Professors Christina Sormani and Anna Sakovich for their interest in this work. Finally, I am sincerely grateful to my thesis advisor, Prof.~Stefano Pigola, for welcoming my research proposal with openness and for engaging in constructive discussions that contributed to its improvement.

\clearpage
\section{Preliminaries}

\subsection{Lorentzian Background}

This subsection provides a brief overview of Lorentzian geometry, presenting definitions and results used throughout. For a comprehensive treatment of the basic theory, see~\cite{oneill}.

\subsubsection{Spacetimes}

A \emph{Lorentzian manifold} is a pair $(M,g)$, where $M$ is a smooth manifold of dimension $n + 1$ and $g$ is a Lorentzian metric on $M$, that is, a metric tensor of signature $(1,n)$. Each non-zero tangent vector $X \in TM$ is classified as \emph{timelike}, \emph{null}, or \emph{spacelike}, according to whether $g(X,X)$ is negative, zero, or positive, respectively. Timelike and null vectors are collectively referred to as \emph{causal}. The $g$-norm of a tangent vector $X \in TM$ is defined by
\[
	\norm{X}_g = \sqrt{\abs{g(X,X)}}.
\]
Causal vectors satisfy a reverse form of the Cauchy--Schwarz inequality.

\begin{proposition}[Reverse Cauchy--Schwarz Inequality]
	Let $X,Y \in TM$ be causal vectors at the same point. Then
	\[
		g(X,Y)^2 \geq g(X,X)g(Y,Y),
	\] 
	with equality if and only if $X$ and $Y$ are linearly dependent.
\end{proposition}

A \emph{time orientation} on a Lorentzian manifold $(M,g)$ is a smooth timelike vector field $\vartheta$ on $M$. A causal vector $X \in TM$ is \emph{future-directed} (respectively, \emph{past-directed}) if $g(X,\vartheta) < 0$ (respectively, $g(X,\vartheta) > 0$). The following lemma ensures that every causal vector must be either future-directed or past-directed. The proof is provided in Appendix~\ref{chap:Deferred_Proofs}.

\begin{lemma}
	\label{lem:causal_timelike}
	Let $X,T \in TM$ be tangent vectors at the same point, with $X$ causal and $T$ timelike. Then $g(X,T) \neq 0$.
\end{lemma}
	
The following characterization of past-directed timelike vectors will be used throughout this work. The proof is provided in Appendix~\ref{chap:Deferred_Proofs}.

\begin{lemma}
	\label{lem:past_timelike}
	Let $p \in M$, and let $h_p$ be any Euclidean inner product on $T_pM$. Then the following conditions on a tangent vector $T \in T_pM$ are equivalent:
	\begin{enumerate}
		\item $T$ is past-directed timelike.
		\item For every future-directed causal vector $X \in T_pM$, $g(X,T) > 0$.
		\item There exists a constant $C > 0$ such that, for all future-directed causal vectors $X \in T_pM$,
		\[
			g(T,X) \geq C\norm{X}_{h_p}.
		\]
	\end{enumerate}
\end{lemma}
	
A \emph{spacetime} is a connected, time-oriented Lorentzian manifold. Unless otherwise specified, the pair $(M,g)$ will denote a spacetime.

\begin{example}
	The most basic example of a spacetime is the $(n+1)$-dimensional \emph{Minkowski spacetime}:
	\[
		\mathbb{M}^{n+1} = (\RR^{n+1}, \eta),
	\]
	where
	\[
		\eta = -dt^2 + d(x^1)^2 + \dots + d(x^n)^2
	\]
	is the Minkowski metric expressed in Cartesian coordinates. The time orientation is given by the coordinate vector field $\partial_t$.
	
	Throughout this work, we denote points in $\mathbb{M}^{n+1}$ by $p = (t_p, x_p)$, where $t_p \in \mathbb{R}$ is the temporal component and $x_p \in \mathbb{R}^n$ the spatial one. Moreover, given two points $p, q \in \mathbb{M}^{n+1}$, we denote by $[p, q]$ the straight-line segment from $p$ to $q$, regarded as a parametrized curve.
\end{example}

\begin{example}
	A \emph{generalized Robertson--Walker (GRW) spacetime} is a spacetime of the form 
	\[
		(M,g) = (I \times S, -dt^2 + f(t)^2h),
	\]
	where $I \subseteq \RR$ is an open interval, $(S,h)$ is an $n$-dimensional connected Riemannian manifold, and $f \from I \to \RR$ is a smooth, strictly positive function called the \emph{scale factor}. The time orientation is given by the coordinate vector field $\partial_t$. 
	
	Note that Minkowski spacetime $\mathbb{M}^{n+1}$ arises as a special case of a GRW spacetime by taking $I = \RR$, $f \equiv 1$, and $(S,h)$ equal to the $n$-dimensional Euclidean space.
\end{example}

\subsubsection{Causality}

A smooth curve $\gamma \from I \to M$ is called \emph{timelike} (respectively, \emph{null}, \emph{spacelike}, \emph{causal}) if its tangent vector $\gamma'$ is everywhere timelike (respectively, null, spacelike, causal). A causal curve is \emph{future-directed} (respectively, \emph{past-directed}) if its tangent vector is everywhere future-directed (respectively, past-directed). 

The \emph{causal future} of a point $p \in M$ is defined by
\[
	J^+(p) = \{ q \in M \mid \text{there exists a future-directed causal curve from }p \text{ to }q \} \cup \{p\}.
\]
Similarly, the \emph{chronological future} of $p$ is given by
\[
	I^+(p) = \{ q \in M \mid \text{there exists a future-directed timelike curve from }p \text{ to }q \}.
\]
The \emph{causal past} $J^-(p)$ and the \emph{chronological past} $I^-(p)$ are defined analogously, replacing future-directed curves with past-directed ones. For any $p,q \in M$, we write $p \leq q$ if and only if $q \in J^+(p)$, or equivalently, $p \in J^-(q)$.

\begin{proposition}
	\label{prop:futures_open}
	For every point $p \in M$, the chronological future $I^+(p)$ and the chronological past $I^-(p)$ are open subsets of $M$.
\end{proposition}

\subsubsection{Time Functions}

A continuous function $\tau \from M \to \RR$ is called a \emph{time function} if it is strictly increasing along all future-directed causal curves. 

Recall that, if a function $f \from M \to \RR$ is differentiable at a point $p \in M$, then its gradient at $p$ is the unique tangent vector satisfying
\[
	g((\nabla f)_p, X) = \mathrm{d}_pf(X) \quad \text{for all } X \in T_pM.
\]
The following lemma shows that if a time function is differentiable at a point, then its monotonicity along future-directed causal curves imposes constraints on the causal character of its gradient. The proof is provided in Appendix~\ref{chap:Deferred_Proofs}.

\begin{lemma}
	\label{lem:time_functions_gradient}
	Let $\tau$ be a time function on $M$. If $\tau$ is differentiable at a point $p \in M$, then the gradient vector $(\nabla \tau)_p$ is either past-directed causal or zero.
\end{lemma}

As a consequence, the gradient of a differentiable time function is either past-directed causal or zero at every point. However, it is not necessarily timelike everywhere. A differentiable function whose gradient is everywhere past-directed timelike is called a \emph{temporal function}. Every temporal function is, in particular, a time function.

\begin{proposition}
	\label{lem:gradient_implies_time_function}
	Let $f \from M \to \RR$ be a temporal function. Then $f$ is a time function.
\end{proposition}
\begin{proof}
	Let $\alpha$ be a future-directed causal curve. Since $\nabla f$ is past-directed timelike and $\alpha'$ is future-directed casual, Lemma~\ref{lem:past_timelike} implies that
	\[
		(f\circ\alpha)' = g(\nabla f, \alpha') > 0.
	\]
	It follows that $f$ is strictly increasing along $\alpha$, and thus a time function. 
\end{proof}

\begin{example}
	\label{ex:GRW_temporal_functions}
	In a GRW spacetime
	\[
		(M,g) = (I \times S, -dt^2 + f(t)^2h),
	\]
	consider a differentiable function of the form $\tau(t,x)= \phi(t)$, where $\phi \from I \to \RR$ satisfies $\phi'(t) > 0$ for all $t \in I$. Then $\tau$ is a temporal function. Indeed, its gradient is given by
	\[
		\nabla \tau = -\phi'(t)\partial_t,
	\]
	which is everywhere past-directed timelike.
\end{example}

We now briefly recall some standard causality conditions. A spacetime $(M,g)$ is said to be \emph{causal} if it contains no closed causal curves, and \emph{stably causal} if this property persists under small perturbations of the metric $g$. Every stably causal spacetime is, in particular, \emph{strongly causal}, meaning that for every point $p \in M$ and every open neighborhood $U$ of $p$, there exists an open neighborhood $V \subseteq U$ of $p$ such that every causal curve with endpoints in $V$ is entirely contained in $U$. The following is a fundamental result in Lorentzian geometry.

\begin{theorem}
	\label{thm:strongly_causal}
	Let $(M,g)$ be a spacetime. Then the following conditions are equivalent:
	\begin{enumerate}
		\item $(M,g)$ admits a time function.
		\item $(M,g)$ admits a smooth temporal function.
		\item $(M,g)$ is stably causal.
	\end{enumerate}
\end{theorem}
\begin{proof}
	See~\cite[Thm.~4.12]{ChruścielGrantMinguzzi2016}.
\end{proof}

\subsubsection{Lorentzian Distance}

The \emph{Lorentzian length} of a causal curve $\gamma \from I \to M$ is defined by
\[
	L_g(\gamma) = \int_I \sqrt{ -g(\gamma'(s), \gamma'(s)) } \, ds.
\]
The \emph{Lorentzian distance} between two points $p,q \in M$ is then given by
\[
	d_g(p,q) = \sup\{ L_g(\gamma) \mid \gamma \text{ is a future-directed causal curve from }p \text{ to }q \}.
\]
If $q = p$ or $q \notin J^+(p)$, we set $d_g(p,q) = 0$. Despite its name, the Lorentzian distance is not a metric, as it may vanish for distinct points and is generally not symmetric, except in trivial cases.

\subsection{Metric Background}

This subsection introduces definitions and results from metric geometry that will be used throughout. We adopt much of the notation and terminology of~\cite{bbi}, to which we also refer for many of the proofs.

\subsubsection{Metric Spaces}

\begin{definition}
	Let $X$ be a set. A function $d \from X \times X \to \RR \cup \{\infty\}$ is called a \emph{metric} on $X$ if the following conditions hold for all $x,y,z \in X$:
	\begin{enumerate}
		\item Positiveness: $d(x,y) \geq 0$, with equality if and only if $x = y$.
		\item Symmetry: $d(x,y) = d(y,x)$.
		\item Triangle inequality: $d(x,z) \leq d(x,y) + d(y,z)$.
	\end{enumerate}	
	The condition that $d(x,y) = 0$ implies $x = y$ is referred to as definiteness. If $d$ satisfies all the conditions above except definiteness, it is called a \emph{semi-metric} on $X$.
\end{definition}

A \emph{metric space} is a pair $(X,d)$, where $X$ is a set and $d$ is a metric on it. Unless multiple metrics on the same set $X$ are considered, we will omit an explicit reference to the metric and write "a metric space $X$" instead of "a metric space $(X,d)$". 

For any $x \in X$ and $r > 0$, we denote by $B_r(x)$ the open ball of radius $r$ centered at $x$, that is, the set of points $y \in X$ such that $d(x,y) < r$. Given a nonempty subset $S \subseteq X$, its diameter is defined by
\[
	\diam(S) = \sup\{ d(x,y) \mid x,y \in S \}.
\]
If $\diam(S) < \infty$, the set $S$ is said to be bounded. Given two nonempty subsets $A,B \subseteq X$, their distance is defined by
\[
	\dist(A,B) = \inf\{ d(a,b) \mid a \in A,\ b \in B \}.
\]

The following result, known as Lebesgue's number lemma, will be used throughout this work.

\begin{proposition}[Lebesgue's Number Lemma]
	Let $X$ be a compact metric space, and let $\{ U_{\alpha} \}_{\alpha \in A}$ be an open cover of $X$. Then there exists $\varepsilon > 0$ such that every ball of radius $\varepsilon$ in $X$ is entirely contained in some $U_{\alpha}$.
\end{proposition} 
\begin{proof}
	See~\cite[Thm.~1.6.11]{bbi}.
\end{proof}

\subsubsection{Hausdorff Distance}

For any nonempty $S \subseteq X$ and $r > 0$, we denote by $U_r(S)$ the tubular open neighborhood of radius $r$ around $S$, that is, the set of points $x \in X$ such that $\dist(x,S) < r$. Equivalently, 
\[
	U_r(S) = \bigcup_{x \in S} B_r(x).
\]
Given two nonempty subsets $A,B \subseteq X$, their \emph{Hausdorff distance} is defined by
\[
	d_H(A,B) = \inf\{ r > 0 \mid A \subseteq U_r(B),\ B \subseteq U_r(A) \}.
\]
It is straightforward to verify the following equivalent expression:
\[
	d_H(A,B) = \max\{ \sup\nolimits_{a \in A}\dist(a,B),\ \sup\nolimits_{b \in B}\dist(b,A) \}.
\]

\begin{proposition}
	Let $X$ be a metric space. Then the Hausdorff distance is a semi-metric on the set of all nonempty subsets of $X$. Moreover, if $A$ and $B$ are closed nonempty subsets of $X$ such that $d_H(A,B) = 0$, then $A = B$. 
\end{proposition}
\begin{proof}
	See~\cite[Prop.~7.3.3]{bbi}.
\end{proof}

We denote by $\mathcal{F}(X)$ the set of all nonempty closed subsets of $X$ (from the French \emph{fermé}), and by $\mathcal{F}_b(X)$ the set of all nonempty closed and bounded subsets of $X$. The previous proposition shows that the Hausdorff distance is a metric on $\mathcal{F}(X)$. Moreover, it is straightforward to verify that $\mathcal{F}_b(X)$ is a closed subset of $\mathcal{F}(X)$.

\begin{lemma}
	Let $X$ be a metric space, and let $\diam \from \mathcal{F}_b(X) \to [0,\infty)$ denote the function that assigns to each set its diameter. Then $\diam$ is $2$-Lipschitz with respect to the Hausdorff distance. 
\end{lemma}
\begin{proof}
	Let $A,B \in \mathcal{F}_b(X)$, and set $d = d_H(A,B)$. By the definition of the Hausdorff distance, for every $\varepsilon > 0$ we have 
	\[
		A \subseteq U_{d + \varepsilon}(B) \quad \text{and} \quad B \subseteq U_{d + \varepsilon}(A).
	\]
	This implies
	\[
		\abs{ \diam(A) - \diam(B) } \leq 2( d_H(A,B) + \varepsilon ).
	\]
	The conclusion follows by letting $\varepsilon \to 0$.
\end{proof}

The following lemma follows directly from the continuity of $\diam$ with respect to the Hausdorff distance and the fact that $\mathcal{F}_b(X)$ is closed in $\mathcal{F}(X)$.

\begin{lemma}
	\label{lem:diam_convergence}
	Let $X$ be a metric space, and let $A_j \in \mathcal{F}_b(X)$ be a sequence of nonempty closed and bounded subsets. If $A_j$ converges to a set $A \in \mathcal{F}(X)$ in the Hausdorff sense, then $A \in \mathcal{F}_b(X)$ and
	\[
		\diam(A) = \lim_{j \to \infty} \diam(A_j).
	\]
\end{lemma}

\subsubsection{Metric Completion}

\begin{proposition}
	Let $X$ be a metric space. Then there exist a complete metric space $\overline{X}$ and an isometric embedding of $X$ as a dense subset of $\overline{X}$. Moreover, any other complete metric space containing an isometric copy of $X$ as a dense subset is isometric to $\overline{X}$.
\end{proposition}
\begin{proof}
	See~\cite[Thm.~1.5.10]{bbi}.
\end{proof}

The space $\overline{X}$ is called the \emph{metric completion} of $X$. In what follows, we identify $X$ with its isometric copy in $\overline{X}$, writing $X \subseteq \overline{X}$ and using the same symbol $d$ to denote both the original metric on $X$ and its extension to $\overline{X}$.

\begin{proposition}
	Let $X$ be a metric space, and let $S \subseteq X$ be a dense subset. Let $Y$ be a complete metric space, and let $f \from S \to Y$ be a $C$-Lipschitz function for some $C > 0$. Then there exists a unique $C$-Lipschitz function $\tilde{f} \from X \to Y$ such that $\at{\tilde{f}}{S} = f$.
\end{proposition}
\begin{proof}
	See~\cite[Prop.~1.5.9]{bbi}.
\end{proof}

\begin{corollary}
	\label{cor:extension_completion}
	Let $f \from X \to Y$ be a $C$-Lipschitz function from a metric space $X$ to a complete metric space $Y$. Then $f$ admits a unique $C$-Lipschitz extension $\overline{f} \from \overline{X} \to Y$ to the metric completion of $X$.
\end{corollary}

\clearpage
\section{Null Distance} 

This section focuses on the notion of null distance, as introduced by Sormani and Vega in~\cite{SormaniVega2016}. Definitions, key results, and proofs are drawn from their work, sometimes adapted for the present context.

\subsection{Definitions}

\begin{definition}
	A curve $\beta \from [a,b] \to M$ is called a \emph{piecewise causal curve} if there exists a partition $a = s_0 < \dots < s_N = b$ such that each restriction $\beta_i = \at{\beta}{[s_i,s_{i+1}]}$ is either a future-directed or a past-directed causal curve. 
\end{definition}

\begin{definition}
	Let $\tau$ be a time function on $M$. The \emph{null length} of a piecewise causal curve $\beta \from [a,b] \to M$, with breakpoints $a = s_0 < \dots < s_N = b$, is defined by
	\[
		\hat{L}_{\tau}(\beta) = \sum_{i=0}^{N-1} \abs{ \tau(\beta(s_{i+1})) - \tau(\beta(s_i)) }.
	\]
	The \emph{null distance} between two points $p,q \in M$ is then given by
	\[
		\hat{d}_{\tau}(p,q) = \inf\{ \hat{L}_{\tau}(\beta) \mid \beta \text{ is a piecewise causal curve from $p$ to $q$} \}.
	\]
	If $q = p$, we set $\hat{d}_{\tau}(p,q) = 0$.
\end{definition}

The following lemma ensures that the null distance between any two points is well-defined.

\begin{lemma}
	\label{lem:null_distance_well-definiteness}
	For any pair of points $p,q \in M$, there exists a piecewise causal curve from $p$ to $q$.
\end{lemma}
\begin{proof}
	Let $\gamma \from [a,b] \to M$ be a smooth curve with $\gamma(a) = p$ and $\gamma(b) = q$. Since 
	\[
		M = \bigcup_{x \in M} I^-(x),
	\]
	and each set $I^-(x)$ is open by Proposition~\ref{prop:futures_open}, the collection $\{ \gamma^{-1}(I^-(x)) \}_{x \in M}$ forms an open cover of the compact interval $[a,b]$. By the Lebesgue's number lemma, there exists $\varepsilon > 0$ such that every subinterval of $[a,b]$ of length less than $\varepsilon$ is contained in some $\gamma^{-1}(I^-(x))$.
	
	Let $a \leq t_0 < \dots < t_{N-1} \leq b$ be a finite collection of points such that the open intervals $( t_i - \varepsilon/2,\ t_i + \varepsilon/2 )$ cover $[a,b]$. For each $i = 1,\dots,N-1$, choose a point $s_i \in [a,b]$ such that
	\[
		t_{i-1} < s_i < t_i, \quad s_i - t_{i-1} < \frac{\varepsilon}{2}, \quad \text{and} \quad t_i - s_i < \frac{\varepsilon}{2}.
	\]
	Define $s_0 = a$ and $s_N = b$. Then each interval $[s_i, s_{i+1}]$ has length less than $\varepsilon$, and is therefore contained in $\gamma^{-1}(I^-(x_i))$ for some $x_i \in M$. It follows that both endpoints $\gamma(s_i)$ and $\gamma(s_{i+1})$ lie in $I^-(x_i)$. Consequently, there exist a future-directed causal curve from $\gamma(s_i)$ to $x_i$ and a past-directed causal curve from $x_i$ to $\gamma(s_{i+1})$. Their concatenation yields a piecewise causal curve from $\gamma(s_i)$ to $\gamma(s_{i+1})$, with a breakpoint at $x_i$. The full concatenation of these segments is a piecewise causal curve from $\gamma(s_0) = p$ to $\gamma(s_N) = q$.
\end{proof}

\subsection{Basic Properties}

\begin{lemma}
	\label{lem:null_distance_basic_properties}
	Let $\tau$ be a time function on $M$, and let $\beta \from [a,b] \to M$ be a piecewise causal curve from $p$ to $q$, with breakpoints $a = s_0 < \dots < s_N = b$. Then:
	\begin{enumerate}
		\item \label{item:1_null_distance_basic_properties} $\hat{L}_{\tau}(\beta) = \tau(q) - \tau(p)$  if and only if $\beta$ is future-directed causal.
		\item \label{item:2_null_distance_basic_properties} The following inequalities hold:
		\[
			\hat{L}_{\tau}(\beta) \geq \max_{s \in [a,b]} \tau(\beta(s)) - \min_{s \in [a,b]} \tau(\beta(s)) \geq \abs{\tau(q) - \tau(p)}.
		\]
		\item \label{item:3_null_distance_basic_properties} If $\tau$ is $C^1$ along $\beta$, then 
		\[
			\hat{L}_{\tau}(\beta) = \int_a^b \abs{(\tau\circ\beta)'} \de s = \sum_{i=0}^{N-1} \int_{s_i}^{s_{i+1}} \abs{(\tau\circ\beta_i)'} \de s.
		\]
	\end{enumerate}
\end{lemma}
\begin{proof}
	\cref{item:1_null_distance_basic_properties} follows directly from the definition of null length and the triangle inequality.
	
	\cref{item:2_null_distance_basic_properties} follows from the triangle inequality and the strict monotonicity of $\tau\circ\beta$ on each subinterval $[s_i, s_{i+1}]$, which implies that its maximum and minimum over $[a,b]$ are attained at some of the breakpoints.
	
	\cref{item:3_null_distance_basic_properties} follows from
	\begin{align*}
		\hat{L}_{\tau}(\beta) &= \sum_{i=0}^{N-1} \abs{ \tau(\beta(s_{i+1})) - \tau(\beta(s_i)) } \\
		&= \sum_{i=0}^{N-1} \abs{ \int_{s_i}^{s_{i+1}} (\tau\circ\beta_i)' \de s } \\
		&= \sum_{i=0}^{N-1} \int_{s_i}^{s_{i+1}} \abs{(\tau\circ\beta_i)'} \de s,
	\end{align*}
	where the last equality holds since $(\tau\circ\beta_i)'$ does not change sign on $[s_i,s_{i+1}]$, due to the causal character of $\beta_i$.
\end{proof}

The following proposition follows readily from the definition of null distance and Lemma~\ref{lem:null_distance_well-definiteness}.

\begin{proposition}
	\label{prop:null_distance_semi-metric}
	Let $\tau$ be a time function on $M$. Then $\hat{d}_{\tau}$ is a semi-metric on $M$.
\end{proposition}

In general, $\hat{d}_{\tau}$ is not a metric, as it may vanish for distinct points (See~\cite[Prop.~3.4]{SormaniVega2016}). We will later introduce the local anti-Lipschitz condition on the time function, which is equivalent to the definiteness of $\hat{d}_{\tau}$. For now, we observe that the following lemma shows that non-definiteness can only occur for pairs of points lying on the same level set of $\tau$.

\begin{lemma}
	\label{lem:time_functions_1Lip}
	Let $\tau$ be a time function on $M$. Then, for every $p,q \in M$,
	\[
		\hat{d}_{\tau}(p,q) \geq \abs{\tau(q) - \tau(p)}.
	\]
	Consequently, definiteness may fail only for pairs of points lying on the same level set of $\tau$:
	\[
		\hat{d}_{\tau}(p,q) = 0 \implies \tau(p) = \tau(q).
	\]
\end{lemma}
\begin{proof}
	Let $\beta$ be a piecewise causal curve from $p$ to $q$. By Lemma~\ref{lem:null_distance_basic_properties}, we have
	\[
		\hat{L}_{\tau}(\beta) \geq \abs{\tau(q) - \tau(p)}.
	\]
	The conclusion follows by taking the infimum over all such curves.
\end{proof}

As a consequence of the previous lemma, we obtain the following causality property of the null distance.
\begin{lemma}
	\label{lem:null_distance_causality_property}
	Let $\tau$ be a time function on $M$. Then $\hat{d}_{\tau}$ satisfies the causality property:
	\[
		p \leq q \implies \hat{d}_{\tau}(p,q) = \tau(q) - \tau(p).
	\]
\end{lemma}
\begin{proof}
	Suppose $p \leq q$, and let $\beta$ be a future-directed causal curve from $p$ to $q$. Since $\hat{L}_{\tau}(\beta) = \tau(q) - \tau(p)$ by Lemma~\ref{lem:null_distance_basic_properties}, it follows that $\hat{d}_{\tau}(p,q)  \leq \tau(q) - \tau(p)$. The reverse inequality follows from Lemma~\ref{lem:time_functions_1Lip}.
\end{proof}

\subsection{Topology}

\begin{lemma}
	\label{lem:diamonds}
	Let $\tau$ be a time function on $M$. Then:
	\begin{enumerate}
		\item \label{item:1_diamonds}$\tau$ is bounded on diamonds: $p \leq x \leq q \implies \tau(p) \leq \tau(x) \leq \tau(q)$.
		\item \label{item:2_diamonds} $\hat{d}_{\tau}$ is bounded on diamonds: $p \leq x, y \leq q \implies \hat{d}_{\tau}(x, y) \leq 2\left( \tau(q) - \tau(p) \right)$.
	\end{enumerate}
\end{lemma}
\begin{proof}
	\cref{item:1_diamonds} follows directly from the fact that $\tau$ is strictly increasing along future-directed causal curves. 
	
	To prove~\cref{item:2_diamonds}, let $\beta$ be a piecewise causal curve consisting of a future-directed causal curve from $x$ to $q$, followed by a past-directed causal curve from $q$ to $y$. Then 
	\[
		\hat{L}_{\tau}(\beta) = \abs{ \tau(q) - \tau(x) } + \abs{ \tau(y) - \tau(q) } \leq 2\left( \tau(q) - \tau(p) \right),
	\]
	where the inequality follows from \cref{item:1_diamonds}.
	Since $\hat{d}_{\tau}(x, y)$ is defined as the infimum of $\hat{L}_{\tau}$ over all piecewise causal curves from $x$ to $y$, the claim follows.
\end{proof}

\begin{lemma}
	Let $\tau$ be a time function on $M$. Then $\hat{d}_{\tau}$ is continuous on $M \times M$.
\end{lemma}
\begin{proof}
	Fix any two points $x,y \in M$. Let $\alpha_x$ be a future-directed timelike curve with $\alpha_x(0) = x$, and let $\alpha_y$ be a future-directed timelike curve with $\alpha_y(0) = y$. Choose $\delta > 0$ sufficiently small so that both curves are defined on the interval $[-\delta, \delta]$. Then, for all $x' \in I^+(\alpha_x(-\delta)) \cap I^-(\alpha_x(\delta))$ and $y' \in I^+(\alpha_y(-\delta)) \cap I^-(\alpha_y(\delta))$, we have
	\begin{align*}
		\abs{ \hat{d}_{\tau}(x,y) - \hat{d}_{\tau}(x',y') } &\leq \hat{d}_{\tau}(x,x') + \hat{d}_{\tau}(y,y') \\
		&\leq 2\left( \tau(\alpha_x(\delta)) - \tau(\alpha_x(-\delta)) \right) + 2\left( \tau(\alpha_y(\delta)) - \tau(\alpha_y(-\delta)) \right),
	\end{align*}
	where the first inequality follows from the reverse triangle inequality, and the second from point~\cref{item:2_diamonds} of Lemma~\ref{lem:diamonds}. Since $\tau$ is continuous, the right-hand side tends to zero as $\delta \to 0$. Hence, $\hat{d}_{\tau}$ is continuous at $(x,y)$.
\end{proof}

As a semi-metric, $\hat{d}_{\tau}$ induces a topology on $M$ generated by the \emph{open null balls}:
\[
	\hat{B}_r(x) = \{ y \in M \mid \hat{d}_{\tau}(x,y) < r \}.
\]

\begin{proposition}
	\label{prop:null_distance_topology}
	Let $\tau$ be a time function on $M$. If $\hat{d}_{\tau}$ is definite, then the topology induced by $\hat{d}_{\tau}$ coincides with the manifold topology.
\end{proposition}
\begin{proof}
	Since $\hat{d}_{\tau}$ is continuous on $M \times M$, every open null ball $\hat{B}_r(x)$ is open in the manifold topology. Hence, the topology induced by $\hat{d}_{\tau}$ is coarser than or equal to the manifold topology.
	
	Conversely, suppose $U \subseteq M$ is open in the manifold topology, and fix any point $x_0 \in U$. Choose a Riemannian metric $h$ on $M$, and denote by $d_h$ the corresponding distance function. Since $U$ is open, there exists $\varepsilon > 0$ such that the Riemannian ball 
	\[
		B = B^h_{\varepsilon}(x_0) = \{ y \in M \mid d_h(x_0,y) < \varepsilon \}
	\]
	is contained in $U$. Since $\hat{d}_{\tau}$ is definite and continuous, and $\partial B$ is compact, the continuous function 
	\[
		z \mapsto \hat{d}_{\tau}(x_0,z)
	\]
	attains a positive minimum on $\partial B$. Define
	\[
		\varepsilon_0 = \min_{z \in \partial B} \hat{d}_{\tau}(x_0,z) > 0.
	\]
	Now let $y \notin B$, and let $\beta$ be any piecewise causal curve from $x_0$ to $y$. Let $z_0 \in \partial B$ be the first point at which $\beta$ intersects $\partial B$, and denote by $\beta_0$ the initial segment of $\beta$ from $x_0$ to $z_0$. Then
	\[
		\hat{L}_{\tau}(\beta) \geq \hat{L}_{\tau}(\beta_0) \geq \varepsilon_0.
	\]
	Taking the infimum over all such curves $\beta$ yields 
	\[
		y \notin B \implies \hat{d}_{\tau}(x_0,y) \geq \varepsilon_0.
	\]
	Thus, we have 
	\[
		\hat{B}_{\varepsilon_0}(x_0) \subseteq B \subseteq U.
	\]
	Since $x_0 \in U$ was arbitrary, $U$ is open in the topology induced by $\hat{d}_{\tau}$. Therefore, the two topologies coincide.
\end{proof}

\subsection{Definiteness}

\subsubsection{Definiteness and Anti-Lipschitz Condition}

\begin{definition}
	A function $f \from M \to \RR$ is said to be \emph{anti-Lipschitz} on a subset $U \subseteq M$ if there exists a metric $d_U$ on $U$ such that, for all $x,y \in U$, 
	\[
		x \leq y \implies f(y) - f(x) \geq d_U(x,y).
	\]
	The function $f$ is called \emph{locally anti-Lipschitz} if it is anti-Lipschitz on some open neighborhood of every point in $M$.
\end{definition}

Since all metrics on a manifold are locally Lipschitz equivalent, the following lemma follows immediately.

\begin{lemma}
	\label{lem:anti-Lip_equivalent_conditions}
	Let $h$ be any Riemannian metric on $M$, and denote by $d_h$ the corresponding distance function. Then the following conditions on a function $f \from M \to \RR$ are equivalent:
	\begin{enumerate}
		\item $f$ is locally anti-Lipschitz.
		\item For every point $p \in M$, there exist an open neighborhood $U$ of $p$ and a constant $C > 0$ such that, for all $x,y \in U$,
		\[
				x \leq y \implies f(y) - f(x) \geq Cd_h(x,y).
		\]
	\end{enumerate}
\end{lemma}

\begin{proposition}
	\label{prop:null_distance_definiteness}
	Let $\tau$ be a time function on $M$. Then $\hat{d}_{\tau}$ is definite if and only if $\tau$ is locally anti-Lipschitz.
\end{proposition}
\begin{proof}
	Suppose first that $\hat{d}_{\tau}$ is definite. Then, by Lemma~\ref{lem:null_distance_causality_property}, it follows that $\tau$ is anti-Lipschitz on all of $M$, and in particular locally anti-Lipschitz.
	
	Conversely, assume that $\tau$ is locally anti-Lipschitz. Let $p,q \in M$ be distinct points, and let $U$ be an open neighborhood of $p$ on which $\tau$ is anti-Lipschitz. Choose a precompact open set $B \subseteq U$ such that $p \in B$ and $q \notin B$, and let $\beta$ be any piecewise causal curve from $p$ to $q$. Let $z_0 \in \partial B$ be the first point at which $\beta$ intersects $\partial B$, and denote by $\beta_0$ the initial segment of $\beta$ from $p$ to $z_0$, with breakpoints $p = x_0, x_1, \dots, x_N = z_0$. Since $\beta_0 \subseteq U$ and $\tau$ is anti-Lipschitz on $U$, we have
	\begin{align*}
		\hat{L}_{\tau}(\beta) &\geq \hat{L}_{\tau}(\beta_0) = \sum_{i=0}^{N-1} \abs{\tau(x_{i+1}) - \tau(x_i)} \\
		&\geq \sum_{i=0}^{N-1} d_U(x_{i+1},x_i) \\
		&\geq d_U(p,z_0) \\
		&\geq \dist(p,\partial B).
	\end{align*}
	Taking the infimum over all such curves $\beta$, it follows that 
	\[
		\hat{d}_{\tau}(p,q) \geq \dist(p,\partial B) > 0,
	\]
	which proves that $\hat{d}_{\tau}$ is definite.
\end{proof}

Combining Propositions~\ref{prop:null_distance_semi-metric},~\ref{prop:null_distance_topology}, and~\ref{prop:null_distance_definiteness}, we obtain the following result.

\begin{theorem}
	Let $\tau$ be a time function on $M$. If $\tau$ is locally anti-Lipschitz, then $\hat{d}_{\tau}$ is a metric on $M$ that induces the manifold topology.
\end{theorem}

\subsubsection{Temporal Functions}

Every temporal function $f \from M \to \RR$ is, in particular, a time function. As such, the associated null distance $\hat{d}_f$ is a semi-metric on $M$. A natural question is whether $\hat{d}_f$ is actually a metric, or equivalently, whether $f$ is locally anti-Lipschitz. This is indeed the case for $C^1$ temporal functions. Moreover, within the $C^1$  class, the local anti-Lipschitz condition characterizes temporal functions.

\begin{theorem}
	\label{thm:temporal_functions_anti-Lip}
	Let $f \from M \to \RR$ be a $C^1$ function. Then $f$ is locally anti-Lipschitz if and only if it is a temporal function.
\end{theorem}

The following proposition provides an equivalent characterization of the local anti-Lipschitz condition for $C^1$ functions.
\begin{proposition}
	\label{prop:C1_functions_anti-Lip}
	Let $h$ be any Riemannian metric on $M$, and denote by $d_h$ the corresponding distance function. Then the following conditions on a $C^1$ function $f \from M \to \RR$ are equivalent:
	\begin{enumerate}
		\item \label{item:1_C1_functions_anti-Lip} For every point $p \in M$, there exist an open neighborhood $U$ of $p$ and a constant $C > 0$ such that, for all $x,y \in U$,
		\[
			x \leq y \implies f(y) - f(x) \geq Cd_h(x,y).
		\]
		\item \label{item:2_C1_functions_anti-Lip} For every point $p \in M$, there exist an open neighborhood $U$ of $p$ and a constant $C > 0$ such that, for all future-directed causal vectors $X \in TU$,
		\[
			g(\nabla f, X) \geq C\norm{X}_h.
		\]
	\end{enumerate}
\end{proposition}
\begin{proof}
	\cref{item:1_C1_functions_anti-Lip}$\implies$\cref{item:2_C1_functions_anti-Lip}. Fix a point $p \in M$, and let $U$ and $C$ be as in~\cref{item:1_C1_functions_anti-Lip}. Let $q \in U$, and	consider any future-directed timelike vector $X \in T_qU$. Let $\alpha \from (-\delta, \delta) \to M$ be an $h$-geodesic with $\alpha(0) = q$ and $\alpha'(0) = X / \norm{X}_h$, for some $\delta > 0$ small enough so that $\alpha(t) \in U$ for all $t \in (-\delta, \delta)$. Since $\alpha'(0)$ is future-directed timelike, the curve $\alpha$ is future-directed timelike near $t = 0$. Furthermore, being an $h$-geodesic, $\alpha$ is locally minimizing for $d_h$. Therefore, for all sufficiently small $t > 0$, we have $\alpha(0) \leq \alpha(t)$ and $d_h(\alpha(0), \alpha(t)) = t$. Applying the inequality in~\cref{item:1_C1_functions_anti-Lip}, we obtain
	\[
		f(\alpha(t)) - f(\alpha(0)) \geq Ct.
	\]
	Dividing both sides by $t$ and taking the limit $t \to 0^+$ yields
	\[
		(f\circ\alpha)'(0) = \lim_{t \to 0^+} \frac{f(\alpha(t)) - f(\alpha(0))}{t} \geq C.
	\]
	Since $(f \circ \alpha)'(0) = g( \nabla f, X / \norm{X}_h )$, we conclude that
	\[
		g(\nabla f, X) \geq C\norm{X}_h.
	\]
	By continuity of $g$ and $h$, this inequality extends to all future-directed causal vectors $X \in TU$. 
	
	\cref{item:2_C1_functions_anti-Lip}$\implies$\cref{item:1_C1_functions_anti-Lip}. Fix a point $p \in M$, and let $U$ and $C$ be as in~\cref{item:2_C1_functions_anti-Lip}. By Lemma~\ref{lem:past_timelike}, the inequality in \cref{item:2_C1_functions_anti-Lip} implies that $\nabla f$ is past-directed timelike on all of $M$. Hence, $f$ is a temporal function. In particular, it is a time function, and thus $M$ is strongly causal by Theorem~\ref{thm:strongly_causal}. Therefore, there exists an open neighborhood $V \subseteq U$ of $p$ such that every future-directed causal curve with endpoints in $V$ is entirely contained in $U$. Let $x,y \in V$ with $x \leq y$, and let $\alpha \from [a,b] \to M$ be a future-directed causal curve from $x$ to $y$. Since $\alpha$ is entirely contained in $U$, we obtain
	\[
		f(y) - f(x) = \int_a^b (f\circ\alpha)' \de s = \int_a^b g(\nabla f, \alpha') \de s \geq CL_h(\alpha) \geq Cd_h(x,y),
	\]
	where the first inequality follows from the assumption in~\cref{item:2_C1_functions_anti-Lip}, and the second from the definition of $d_h(x,y)$ as the infimum of the $h$-length over all smooth curves from $x$ to $y$.
\end{proof}

\begin{definition}
	\label{def:bounded_away_zero}
	Let $h$ be any Riemannian metric on $M$ and let $\mathcal{T} \subseteq TM$ be a collection of tangent vectors. We say that $\mathcal{T}$ is \emph{locally bounded away from zero} if, for every point $p \in M$, there exist an open neighborhood $U$ of $p$ and a constant $C > 0$ such that, for all $T \in \mathcal{T} \cap TU$, $\norm{T}_g \geq C$.
\end{definition}

\begin{lemma}
	\label{lem:bounded_away_zero}
	Let $h$ be any Riemannian metric on $M$, and let $\mathcal{T} \subseteq TM$ be a collection of tangent vectors. If the vectors in $\mathcal{T}$ are past-directed timelike and locally bounded away from zero, then, for every point $p \in M$, there exist an open neighborhood $U$ of $p$ and a constant $C > 0$ such that, for all future-directed causal vectors $X \in TU$ and all $T \in \mathcal{T} \cap TU$, 
	\[
		g(T,X) \geq C\norm{X}_h.
	\]
\end{lemma}
\begin{proof}
	Suppose, for contradiction, that the conclusion fails at some point $p_0 \in M$. Let $U$ and $C $ be as in Definition~\ref{def:bounded_away_zero}, corresponding to $p = p_0$. Then there exist a sequence of points $p_j \in U$, with $p_j \to p_0$, a sequence of future-directed causal vectors $X_j \in T_{p_j}M$, and a sequence of past-directed timelike vectors $T_j \in \mathcal{T} \cap T_{p_j}M$ such that
	\[
		g(T_j, X_j) < j^{-1}\norm{X_j}_h.
	\]
	By rescaling and passing to a subsequence if necessary, we may assume that $\norm{X_j}_h = 1$ for all $j$ and that $X_j \to X_0$ for some future-directed causal vector $X_0 \in T_{p_0}M$. Assume first that the sequence $T_j$ is bounded in the $h$-norm. Then, up to a further subsequence, $T_j \to V_0 \in T_{p_0}M$, where $V_0$ is either past-directed causal or zero. Passing to the limit in the inequality yields
	\[
		g(V_0, X_0) = 0.
	\]
	By Lemma~\ref{lem:past_timelike}, it follows that $V_0$ is either null or zero. In both cases, we have $\norm{T_j}_g \to \norm{V_0}_g = 0$, contradicting the assumption $\norm{T_j}_g \geq C > 0$.	Suppose instead that $\norm{T_j}_h \to \infty$. Define $W_j = T_j / \norm{T_j}_h$. Then, for sufficiently large $j$, we have $\norm{T_j}_h \geq 1$, and thus
	\[
		g(W_j, X_j) < j^{-1}\norm{T_j}_h^{-1} \leq j^{-1}.
	\]
	Since $\norm{W_j}_h = 1$ for all $j$, we may extract a subsequence such that $W_j \to W_0$ for some past-directed causal vector $W_0 \in T_{p_0}M$. Passing to the limit in the inequality yields 
	\[
		g(W_0, X_0) = 0.
	\]
	By Lemma~\ref{lem:past_timelike}, it follows that $V_0$ is null. This contradicts the assumption, since for sufficiently large $j$ we have
	\[
		0 < C \leq \frac{\norm{T_j}_g}{\norm{T_j}_h} = \norm{W_j}_g \to \norm{W_0}_g = 0.
	\]
\end{proof}

We are now ready to prove Theorem~\ref{thm:temporal_functions_anti-Lip}.

\begin{proof}[Proof of Theorem~\ref{thm:temporal_functions_anti-Lip}]
	Suppose first that $f$ is locally anti-Lipschitz, and let $h$ be any Riemannian metric on $M$. Then, by Proposition~\ref{prop:C1_functions_anti-Lip} and Lemma~\ref{lem:anti-Lip_equivalent_conditions}, for every point $p \in M$, there exist an open neighborhood $U$ of $p$ and a constant $C > 0$ such that, for all future-directed causal vectors $X \in TU$,
	\[
		g(\nabla f, X) \geq C\norm{X}_h.
	\]
	By Lemma~\ref{lem:past_timelike}, it follows that $\nabla f$ is past-directed timelike on all of $M$, and therefore $f$ is a temporal function.
	
	Conversely, suppose that $f$ is a temporal function. Then $\nabla f$ is a continuous past-directed timelike vector field on $M$. By continuity, the collection $\{ (\nabla f)_p \}_{p \in M}$ is locally bounded away from zero. It then follows from Lemma~\ref{lem:bounded_away_zero}, Proposition~\ref{prop:C1_functions_anti-Lip}, and Lemma~\ref{lem:anti-Lip_equivalent_conditions} that $f$ is locally anti-Lipschitz.
\end{proof}

\clearpage
\section{Main Result}

The purpose of this section is to establish Theorem~\ref{thm:inequality}. We begin by analyzing a guiding example in the two-dimensional Minkowski spacetime.

\begin{example}
	\label{ex:Minkowski_inequality}
	Consider the temporal function $f(t,x) = t$ on the two-dimensional Minkowski spacetime $\mathbb{M}^2$. For any fixed $t \in \RR$, let $p,q \in M_t$ be two distinct points lying on the same level set
	\[
		M_t = f^{-1}(t) = \{ (t',x') \in \RR^2 \mid t' = t \},
	\]
	and let $d = \abs{ x_q - x_p }$ denote the Euclidean distance between their spatial components.	Since the induced Riemannian metric $h_t$ on $M_t$ coincides with the standard Euclidean metric, it follows that
	\[
		d_{h_t}(p,q) = d.
	\]
	Now consider the point 
	\[
		r = \left( t - \frac{d}{2}, \frac{x_p + x_q}{2} \right), 
	\]
	which lies at the intersection of the past null cones of $p$ and $q$. Define the piecewise causal curve $\beta$ as the concatenation of the null segments $[p, r]$ and $[r, q]$. Then 
	\[
		\hat{L}_f(\beta) = 2\abs{t_r - t} = d.
	\]
	Therefore,
	\[
		\frac{\hat{d}_f(p,q)}{d_{h_t}(p,q)} \leq \frac{\hat{L}_f(\beta)}{d_{h_t}(p,q)} = 1.
	\]
\end{example}

The key idea in the previous example is the construction of a piecewise causal curve whose null length equals the spatial distance between its endpoints. The strategy underlying the proof of Theorem~\ref{thm:inequality} is to asymptotically reproduce this construction, as formalized in the following lemma.

\begin{lemma}
	\label{lem:local_curve_construction}
	Let $(M,g)$ be a spacetime admitting a smooth temporal function $f \from M \to \RR$. Suppose there exist $t \in \RR$ and $C > 0$ such that the level set $M_t = f^{-1}(t)$ is nonempty and the equality 
	\[
		g(\nabla f, \nabla f) = -C^2
	\]
	holds at every point of $M_t$. Fix a point $p \in M_t$, and let $\gamma \from [0,L] \to M_t$ be a smooth unit-speed curve starting at $p$. Then, for every $\varepsilon \in (0,1)$ and every $s \in (0,L]$, there exists a piecewise causal curve $\beta_s$ from $p$ to $\gamma(s)$ such that\footnote{The curve $\beta_s$ depends on both $s$ and $\varepsilon$, but only the dependence on $s$ is made explicit to simplify the notation.}
	\[
		\lim_{\varepsilon \to 0^+} \lim_{s \to 0^+} \frac{\hat{L}_f(\beta_s)}{L_{h_t}(\at{\gamma}{[0,s]})} = C,
	\]
	where $h_t$ denotes the induced Riemannian metric on $M_t$ and $L_{h_t}$ is the associated length functional. Consequently,
	\[
		\lim_{s \to 0^+} \frac{\hat{d}_f(p, \gamma(s))}{L_{h_t}(\at{\gamma}{[0,s]})} \leq C.
	\]
\end{lemma}

To prove this lemma, we first establish a preliminary result showing that, within a sufficiently small normal coordinate neighborhood, piecewise causal curves for a suitably perturbed Minkowski metric are mapped to piecewise causal curves for the spacetime metric. 

For each $\varepsilon \in (0,1)$, define the \emph{$\varepsilon$-slim Minkowski metric} on $\RR^{n+1}$ by 
\[
	\eta_{\varepsilon} = -(1-\varepsilon)dt^2 + d(x^1)^2 + \dots + d(x^n)^2,
\]
and denote by $C_{\varepsilon}$ the corresponding causal cone:
\[
	C_{\varepsilon} = \{ v \in \RR^{n+1} \setminus \{ 0 \} \mid \eta_{\varepsilon}(v,v) \leq 0 \}. 
\]
We then have the following result.

\begin{lemma}
	\label{lem:slim_Minkowski}
	Let $(U,\varphi)$ be a normal coordinate chart centered at $p \in M$. Then, for every $\varepsilon \in (0,1)$, there exists an open neighborhood $U_{\varepsilon} \subseteq U$ of $p$ such that every piecewise $\eta_{\varepsilon}$-causal curve in $\varphi(U_{\varepsilon})$ is mapped by $\varphi^{-1}$ to a piecewise $g$-causal curve in $U_{\varepsilon}$.
\end{lemma}
\begin{proof}	
	At the origin $0 \in \RR^{n+1}$, the pullback of $g$ under $\varphi^{-1}$ satisfies\footnote{We identify the tangent space at any point of $\RR^{n+1}$ with $\RR^{n+1}$ itself.}
	\[
		((\varphi^{-1})^*g)_0(v,v) = \eta(v,v) < \eta_{\varepsilon}(v,v) 
	\]
	for all $v \in C_{\varepsilon}$. In particular, the inequality holds on the compact set 
	\[
		D_{\varepsilon} = S^n \cap C_{\varepsilon},
	\]
	where $S^n$ denotes the unit sphere in $\RR^{n+1}$.	Define the function
	\[
		F \from \varphi(U) \times D_{\varepsilon} \to \RR, \quad F(x,v) = ((\varphi^{-1})^*g)_x(v,v) - \eta_{\varepsilon}(v,v).
	\]
	By continuity of $F$ and the fact that $F(0,v) < 0$ for all $v \in D_{\varepsilon}$, for every $v \in D_{\varepsilon}$ there exist open neighborhoods $U_v \subseteq \varphi(U)$ of $0$ and $V_v \subseteq D_{\varepsilon}$ of $v$ such that $F(x,w) < 0$ for all $(x,w) \in U_v \times V_v$. By compactness of $D_{\varepsilon}$, there exist finitely many points $v_1,\dots,v_N \in D_{\varepsilon}$ such that 
	\[
		D_{\varepsilon} \subseteq \bigcup_{i=1}^N V_{v_i}.
	\]
	For each $i=1,\dots,N$, let $r_{i} > 0$ be such that the Euclidean open ball $B_{r_i}(0)$ is contained in $U_{v_i}$, and define
	\[
		r_{\min} = \min_{1 \leq i \leq N} r_i, \quad U_{\varepsilon} = \varphi^{-1}( B_{r_{\min}}(0) ).
	\]
	Fix $x \in \varphi(U_{\varepsilon})$ and $w \in C_{\varepsilon}$. Then the normalized vector $\hat{w} = w/\abs{w}$ lies in $D_{\varepsilon}$, hence in some $V_{v_i}$, and $x \in B_{r_i}(0) \subseteq U_{v_i}$. Therefore,
	\[
		F(x,w) = \abs{w}^2 F(x, \hat{w}) < 0,
	\]
	which implies 
	\[
		((\varphi^{-1})^*g)_x(w,w) < \eta_{\varepsilon}(w,w).
	\]
	Thus, $\mathrm{d}_x(\varphi^{-1})$ maps $\eta_{\varepsilon}$-causal vectors to $g$-causal vectors. Consequently, $\varphi^{-1}$ maps piecewise $\eta_{\varepsilon}$-causal curves in $\varphi(U_{\varepsilon})$ to piecewise $g$-causal curves in $U_{\varepsilon}$.
\end{proof}

We now proceed with the proof of Lemma~\ref{lem:local_curve_construction}.

\begin{proof}[Proof of Lemma~\ref{lem:local_curve_construction}]
	Let $(U,\varphi)$ be a normal coordinate chart centered at $p$. Since $(\nabla f)_p$ is a past-directed timelike vector satisfying 
	\[
		g_p((\nabla f)_p, (\nabla f)_p) = -C^2,
	\]
	and $\gamma'(0)$ is a spacelike unit vector, we may, up to a Lorentz transformation, assume that
	\begin{equation}
		\label{eqn:e_0}
		\mathrm{d}_p\varphi((\nabla f)_p) = -Ce_0
	\end{equation}
	and
	\begin{equation}
		\label{eqn:e_1}
		\mathrm{d}_p\varphi(\gamma'(0)) = e_1,
	\end{equation}
	where $\{ e_i \}_{i=0}^n$ denotes the standard orthonormal basis of the Minkowski spacetime $\mathbb{M}^{n+1}$. Fix $\varepsilon \in (0,1)$, and let $U_{\varepsilon} \subseteq U$ be the open neighborhood provided by Lemma~\ref{lem:slim_Minkowski}. For each $s \in (0,L]$, define $q_s = \varphi(\gamma(s))$ and write $q_s = (t_s,x_s)$, where $t_s \in \RR$ and $x_s \in \RR^n$ denote the temporal and spatial components of $q_s$, respectively. Equation~\eqref{eqn:e_1} ensures that $\abs{x_s} \neq 0$ for sufficiently small $s$. Furthermore, the first-order Taylor expansion of $\varphi\circ\gamma$ at zero yields
	\begin{equation}
		\label{eqn:expansions}
		\abs{x_s} = s + o(s), \quad t_s = o(s) \quad \text{as }\,s \to 0^+.
	\end{equation}
	It follows that, for sufficiently small $s$, 
	\[
		\abs{x_s}^2 - (1-\varepsilon)t_s^2 \neq 0,
	\]
	ensuring that the quantity
	\[
		t^*_s = \frac{\abs{x_s}^2 - (1-\varepsilon)t_s^2}{2\sqrt{1-\varepsilon} \left(  \abs{x_s} - \sqrt{1-\varepsilon}\, t_s \right)}.
	\]
	is well-defined and nonzero. Consequently, for sufficiently small $s$, the point
	\[
		 r_s = \left( t^*_s, \sqrt{1-\varepsilon}\, t^*_s\frac{x_s}{\abs{x_s}} \right),
	\]
	is well-defined and nonzero. Substituting the expansions in~\eqref{eqn:expansions} into the expression for $t^*_s$ and simplifying by absorbing higher-order terms into $o(s)$, we find
	\[
		t^*_s = \frac{1}{2\sqrt{1 - \varepsilon}} (s + o(s)) \quad \text{as }\,s \to 0^+,
	\]
	and thus,
	\[
		\lim_{s \to 0^+} \frac{t^*_s}{s} = \frac{1}{2\sqrt{1 - \varepsilon}}.
	\]	
	A direct computation shows that the straight-line segments $[0, r_s]$ and $[r_s, q_s]$ are $\eta_{\varepsilon}$-null. Furthermore, the Euclidean norm of $r_s$ satisfies
	\begin{equation}
		\label{eqn:norm_rs}
		\abs{r_s} = \sqrt{2 - \varepsilon}\, \abs{t^*_s} = \frac{1}{2}\sqrt{\frac{2 - \varepsilon}{1 - \varepsilon}}\, \abs{ s + o(s) } \quad \text{as }\,s \to 0^+.
	\end{equation}
	Define the piecewise $\eta_{\varepsilon}$-null curve $\tilde{\beta}_s$ as the concatenation of the segments $[0, r_s]$ and $[r_s, q_s]$. Since $\varphi(U_{\varepsilon})$ is open and both $r_s$ and $q_s$ converge to $0$ as $s \to 0^+$, the curve $\tilde{\beta}_s$ is entirely contained in $\varphi(U_\varepsilon)$ for sufficiently small $s$. By Lemma~\ref{lem:slim_Minkowski}, the curve $\beta_s = \varphi^{-1}\circ\tilde{\beta}_s$ is piecewise $g$-causal in $U_{\varepsilon}$, and its null length satisfies
	\begin{align*}
		\hat{L}_f(\beta_s) &= 2\abs{ (f\circ\varphi^{-1})(r_s) - (f\circ\varphi^{-1})(0) } \\
		&= 2\abs{ (\nabla(f\circ\varphi^{-1}))_0 \cdot r_s + o(\abs{r_s}) } \quad \text{as }\,\abs{r_s} \to 0,
	\end{align*}
	where the second equality follows from the first-order Taylor expansion of $f\circ\varphi^{-1}$ at the origin. Moreover, using the definition of the gradient, equation~\eqref{eqn:e_0}, and the fact that $\varphi$ is a normal coordinate chart, we obtain
	\begin{align*}
		(\nabla(f\circ\varphi^{-1}))_0 \cdot r_s &= \mathrm{d}_0(f\circ\varphi^{-1})(r_s) = \mathrm{d}_pf ( \mathrm{d}_0\varphi^{-1}(r_s) ) \\
		&= g_p( (\nabla f)_p, \mathrm{d}_0\varphi^{-1}(r_s) ) = \eta( \mathrm{d}_p\varphi((\nabla f)_p), r_s ) \\
		&= \eta(-Ce_0, r_s) = Ct^*_s.
	\end{align*}
	It follows that
	\[
		\hat{L}_f(\beta_s) = 2\abs{ Ct^*_s + o(\abs{r_s}) } \quad \text{as }\,\abs{r_s} \to 0.
	\]
	Since $\gamma$ is unit-speed, we have $L_{h_t}(\at{\gamma}{[0,s]}) = s$. Therefore,
	\begin{equation}
		\label{eqn:ratio}
		\frac{ \hat{L}_f(\beta_s) }{ L_{h_t}(\at{\gamma}{[0,s]}) } = \frac{2\abs{ Ct^*_s + o(\abs{r_s}) }}{s} \quad \text{as }\,\abs{r_s} \to 0.
	\end{equation}
	Using~\eqref{eqn:norm_rs}, we deduce that
	\[
		\lim_{s \to 0^+} \frac{o(\abs{r_s})}{s} = 0.
	\]
	Taking the limits $s \to 0^+$ and then $\varepsilon \to 0^+$ in~\eqref{eqn:ratio}, we conclude that
	\[
		\lim_{\varepsilon \to 0^+} \lim_{s \to 0^+} \frac{ \hat{L}_f(\beta_s) }{ L_{h_t}(\at{\gamma}{[0,s]}) } = \lim_{\varepsilon \to 0^+} \lim_{s \to 0^+} \frac{2\abs{ Ct^*_s + o(\abs{r_s}) }}{s} = \lim_{\varepsilon \to 0^+} \frac{C}{\sqrt{1-\varepsilon}} = C.
	\]
\end{proof}

We now complete the argument by establishing Theorem~\ref{thm:inequality}.
\begin{proof}[Proof of Theorem~\ref{thm:inequality}]
    Let $\delta > 0$, and let $\gamma : [0,L] \to M_t$ be a smooth unit-speed curve from $p$ to $q$ satisfying
    \begin{equation}
    	\label{eqn:almost_min_curve}
        L_{h_t}(\gamma) < d_{h_t}(p,q) + \delta.
    \end{equation}
    By Lemma~\ref{lem:local_curve_construction}, for every $s \in [0,L]$, there exists $\varepsilon_s > 0$ such that, for all $s' \in [0,L]$ with $\abs{s - s'} < \varepsilon_s$, we have
    \[
        \hat{d}_f(\gamma(s),\gamma(s')) \leq (C+\delta)L_{h_t}(\at{\gamma}{[s,s']}).
    \]
    The collection of intervals $\{ (s - \varepsilon_s/2, s + \varepsilon_s/2) \}_{s \in [0,L]}$ forms an open cover of $[0,L]$. Using the Lebesgue's number lemma, we can find a partition $0 = s_0 < \dots < s_N = L$ such that, for each $i = 0, \dots, N-1$,
    \begin{equation}
    	\label{eqn:local_estimates}
        \hat{d}_f(\gamma(s_i),\gamma(s_{i+1})) \leq (C+\delta)L_{h_t}(\at{\gamma}{[s_i,s_{i+1}]}).
    \end{equation}
    Applying the triangle inequality together with estimates~\eqref{eqn:local_estimates} and~\eqref{eqn:almost_min_curve}, we obtain
    \begin{align*}
        \hat{d}_f(p,q) &\leq \sum_{i=0}^{N-1} \hat{d}_f(\gamma(s_i),\gamma(s_{i+1})) \\
        &\leq (C + \delta)\sum_{i=0}^{N-1} L_{h_t}(\at{\gamma}{[s_i,s_{i+1}]}) \\
        &= (C + \delta)L_{h_t}(\gamma) \\
        &< (C + \delta)(d_{h_t}(p,q) + \delta).
    \end{align*}
    The conclusion follows by letting $\delta \to 0$.
\end{proof}

As a direct consequence of Theorem~\ref{thm:inequality} and Example~\ref{ex:GRW_temporal_functions}, we obtain the following result.

\begin{corollary}
	In a GRW spacetime
	\[
		(M,g) = (I \times S, -dt^2 + f(t)^2h),
	\]
	consider a smooth temporal function of the form $\tau(t,x)= \phi(t)$, where $\phi \from I \to \RR$ satisfies $\phi'(t) > 0$ for all $t \in I$. For any $t \in I$, consider the level set $M_t = \tau^{-1}(t)$. Then, for every $p,q \in M_t$, 
	\[
		\hat{d}_{\tau}(p,q) \leq \phi'(t)d_{h_t}(p,q),
	\]
	where $h_t$ denotes the induced Riemannian metric on $M_t$ and $d_{h_t}$ is the corresponding distance function. 
\end{corollary}

\clearpage
\section{Cosmological Time}

This section focuses on the notion of regular cosmological time, as introduced by Andersson, Galloway, and Howard in~\cite{AnderssonGallowayHoward1998}. It also presents the proof of Theorem~\ref{thm:initial_level_set}.

\subsection{Regular Cosmological Time}

The \emph{cosmological time function} $\tau_g \from M \to [0,\infty]$ is defined by
\[
	\tau_g(q) = \sup_{p \leq q} d_g(p,q).
\]
In general, the function $\tau_g$ is not necessarily well-behaved; for instance, it is identically infinite in Minkowski spacetime. In \cite{AnderssonGallowayHoward1998}, Andersson, Galloway, and Howard introduced a notion of regularity for the cosmological time function and proved that, under this condition, several important properties hold. A future-directed causal curve $\gamma \from (a,b) \to M$ is said to be \emph{past inextendible} if there exists no point $p \in M$ such that 
\[
	\lim_{s \to a^+} \gamma(s) = p.
\]
The cosmological time function $\tau_g$ is called \emph{regular} if it satisfies the following conditions:
\begin{enumerate}
	\item $\tau_g(q) < \infty$ for all $q \in M$;
	\item $\tau_g \to 0$ along all past inextendible future-directed causal curves.
\end{enumerate}

\begin{proposition}[{\cite[Thm.~1.2]{AnderssonGallowayHoward1998}}]
	\label{prop:generators}
	Suppose that the cosmological time function $\tau_g$ is regular. Then the following properties hold:
	\begin{enumerate}
		\item \label{item:1_generators} $\tau_g$ is a time function, and for all $x,y \in M$,
		\[
			x \leq y \implies \tau_g(y) - \tau_g(x) \geq d_g(x,y).
		\]
		\item The gradient $\nabla \tau_g$ exists almost everywhere.
		\item \label{item:3_generators} For every point $q \in M$, there exists a future-directed timelike unit-speed geodesic $\gamma_q \from (0,\tau_g(q)] \to M$ such that $\gamma_q(\tau_g(q)) = q$ and $\tau_g(\gamma_q(t)) = t$ for all $t \in (0,\tau_g(q))$. 
	\end{enumerate}
\end{proposition}

A curve as in point~\cref{item:3_generators} above is called a \emph{generator} at $q$. The following result shows that, wherever it exists, the gradient of a regular $\tau_g$ is a past-directed timelike unit vector.

\begin{proposition}
	\label{prop:cosmo_time_unit_norm}
	Suppose that the cosmological time function $\tau_g$ is regular and differentiable at a point $q \in M$. Then there exists a unique generator $\gamma_q$ at $q$, and $(\nabla \tau_g)_q = -\gamma_q'(\tau_g(q))$. 
\end{proposition}
\begin{proof}
	Set $\tau = \tau_g$, and let $\gamma = \gamma_q$ be any generator at $q$. Define $u = \gamma'(\tau(q))$, which is a future-directed timelike unit vector at $q$. Then the gradient of $\tau$ at $q$ can be expressed as 
	\[
		(\nabla \tau)_q = -au + bv,
	\]
	where $v \in T_qM$ is a spacelike unit vector orthogonal to $u$, and $a,b \geq 0$. Since $\gamma$ is a generator at $q$, we have
	\[
		a = g((\nabla \tau)_q, u) = \frac{\mathrm{d}}{\mathrm{d}t}(\tau\circ\gamma) = \frac{\mathrm{d}}{\mathrm{d}t}(t) = 1.
	\]
	Thus, $(\nabla \tau)_q = -u + bv$. By Lemma~\ref{lem:time_functions_gradient}, the vector $(\nabla \tau)_q$ must be causal, which implies $0 \leq b \leq 1$. In particular,
	\[
		\norm{(\nabla \tau)_q}_g = \sqrt{1 - b^2} \leq 1.
	\]
	Now fix any future-directed timelike unit vector $V \in T_qM$, and let $\alpha \from (-\delta,\delta) \to M$ be a future-directed timelike unit-speed curve with $\alpha(0) = q$ and $\alpha'(0) = V$. By point~\cref{item:1_generators} of Proposition~\ref{prop:generators}, we obtain 
	\[
		g((\nabla \tau)_q, V) = \lim_{t \to 0^+} \frac{\tau(\alpha(t)) - \tau(\alpha(0))}{t} \geq \lim_{t \to 0^+} \frac{d_g(\alpha(0), \alpha(t))}{t} \geq 1.
	\]
	We now show that this forces $b = 0$, hence $(\nabla \tau)_q = -u$. First, suppose $b \neq 1$, so that $(\nabla \tau)_q$ is not null. Applying the inequality above to $V = -(\nabla \tau)_q / \norm{(\nabla \tau)_q}_g$ yields 
	\[
		\norm{(\nabla \tau)_q}_g = g((\nabla \tau)_q, V) \geq 1.
	\]
	Combining this with the earlier bound $\norm{(\nabla \tau)_q}_g \leq 1$, we conclude that $\norm{(\nabla \tau)_q}_g = 1$, and thus $b = 0$. Now suppose $b = 1$. For every $A \geq 1$, consider the future-directed timelike unit vector 
	\[
		V = Au - \sqrt{A^2 - 1}\, v.
	\]
	Then
	\[
		1 \leq g((\nabla \tau)_q, V) = A - \sqrt{A^2 - 1}.
	\]
	However, the right-hand side tends to $0$ as $A \to \infty$, yielding a contradiction. Therefore, the case $b = 1$ is excluded, and from the previous argument we deduce that $b = 0$.
	
	Finally, the uniqueness of the generator at $q$ follows from the uniqueness of geodesics with given initial data.
\end{proof}

\subsection{Initial Level Set}

As a consequence of Proposition~\ref{prop:cosmo_time_unit_norm}, a smooth regular cosmological time function $\tau_g$ is a smooth temporal function with unit gradient everywhere. Then the associated null distance $\hat{d}_{\tau_g}$ is a metric on $M$ that induces the manifold topology. Furthermore, by Lemma~\ref{lem:time_functions_1Lip}, $\tau_g$ is $1$-Lipschitz with respect to $\hat{d}_{\tau_g}$, and by Corollary~\ref{cor:extension_completion}, it admits a unique $1$-Lipschitz extension $\overline{\tau_g} \from \overline{M} \to [0,\infty)$ to the metric completion of $(M, \hat{d}_{\tau_g})$.

\begin{lemma}
	\label{lem:generators_1Lip}
	Suppose that the cosmological time function $\tau_g$ is smooth and regular. Then, for every $q \in M$, the unique generator $\gamma_q \from (0,\tau_g(q)] \to M$ at $q$ is $1$-Lipschitz with respect to $\hat{d}_{\tau_g}$. In particular, $\gamma_q$ admits a unique $1$-Lipschitz extension $\overline{\gamma}_q \from [0,\tau_g(q)] \to \overline{M}$.
\end{lemma}
\begin{proof}
	For all $t,t' \in (0,\tau_g(q)]$, we have
	\[
		\hat{d}_{\tau_g}(\gamma_q(t), \gamma_q(t')) \leq \hat{L}_{\tau_g}(\at{\gamma_q}{[t,t']}) = \abs{ \tau_g(t) - \tau_g(t') } = \abs{t - t'}
	\]
	where we used Lemma~\ref{lem:null_distance_basic_properties} and the fact that $\gamma_q$ is a generator. This shows that $\gamma_q$ is $1$-Lipschitz with respect to $\hat{d}_{\tau_g}$. The existence and uniqueness of the extension then follow from Corollary~\ref{cor:extension_completion}.
\end{proof}

We are now ready to prove Theorem~\ref{thm:initial_level_set}.

\begin{proof}[Proof of Theorem~\ref{thm:initial_level_set}]
	We first prove that $\overline{\tau_g}^{-1}(0) \neq \emptyset$. Fix a point $q \in M$, and let $\gamma_q \from (0,\tau_g(q)] \to M$ be the unique generator at $q$. Consider a sequence $t_k \in (0,\tau(q)]$ with $t_k \to 0$. Then the sequence $\gamma_q(t_k) \in M$ is Cauchy with respect to $\hat{d}_{\tau_g}$: 
	\[
		\hat{d}_{\tau_g}(\gamma_q(t_k), \gamma_q(t_j)) \leq \abs{t_k - t_j} \to 0,
	\]
	where we used Lemma~\ref{lem:generators_1Lip}. Since $\overline{M}$ is complete, the sequence converges to a point $q_0 \in \overline{M}$. By continuity of $\overline{\tau_g}$ and the fact that $\gamma_q$ is a generator, we obtain
	\[
		\overline{\tau_g}(q_0) = \lim_k \tau_g(\gamma_q(t_k)) = \lim_k t_k = 0,
	\]
	which shows that $q_0 \in \overline{\tau_g}^{-1}(0)$. 
	
	We now prove that $\overline{\tau_g}^{-1}(0)$ consists of a single point. Since $\tau_g$ is a smooth temporal function with unit gradient everywhere, Theorem~\ref{thm:inequality} implies that for any $t > 0$ such that $M_t$ is nonempty, and for every $p,q \in M_t$, 
	\[
		\hat{d}_{\tau_g}(p,q) \leq d_{h_t}(p,q).
	\]
	In particular, it follows that
	\[
		\lim_{t \to 0^+} \diam_{\hat{d}_{\tau_g}}(M_t) = 0.
	\]
	Since, for sufficiently small $t > 0$, $M_t$ is a nonempty closed and bounded subset of $\overline{M}$, and $\overline{\tau_g}^{-1}(0)$ is nonempty and closed, Lemma~\ref{lem:diam_convergence} applies; hence, it suffices to show that $M_t$ converges to $\overline{\tau_g}^{-1}(0)$ in the Hausdorff sense as $t \to 0$. Indeed, in that case,
	\[
		\diam(\overline{\tau_g}^{-1}(0)) = \lim_{t \to 0^+} \diam_{\hat{d}_{\tau_g}}(M_t) = 0,
	\]
	which implies that $\overline{\tau_g}^{-1}(0)$ consists of a single point. For every $t > 0$ such that $M_t$ is nonempty, we have
	\[
		d_{H}( M_t, \overline{\tau_g}^{-1}(0) ) = \max\{ \sup\nolimits_{q_t \in M_t}\dist(q_t, \overline{\tau_g}^{-1}(0)),\ \sup\nolimits_{q_0 \in \overline{\tau_g}^{-1}(0)}\dist(q_0, M_t) \}.
	\]
	We estimate each term separately. Fix any point $q_t \in M_t$, and let $\gamma_{q_t} \from (0,t] \to M$ be the unique generator at $q_t$. Let $\overline{\gamma}_{q_t} \from [0,t] \to \overline{M}$ be its $1$-Lipschitz extension. Then
	\begin{equation}
		\label{eqn:t_estimate_1}
		\dist(q_t, \overline{\tau_g}^{-1}(0)) \leq \hat{d}_{\tau_g}(q_t, \overline{\gamma}_{q_t}(0)) = \hat{d}_{\tau_g}(\overline{\gamma}_{q_t}(t), \overline{\gamma}_{q_t}(0)) \leq t.
	\end{equation}
	Now fix any point $q_0 \in \overline{\tau_g}^{-1}(0)$, and let $q_i \in M$ be a sequence converging to $q_0$. Choose any point $q \in M_t$, and let $\gamma_q \from (0,t] \to M$ be its generator. Given $\varepsilon > 0$, choose $0 < \tilde{t} \leq t$ such that
	\[
		\diam_{\hat{d}_{\tau_g}}(M_{t'}) < \varepsilon/2 \quad \text{for all } 0 < t' \leq \tilde{t}.
	\]
	Since $q_i \to q_0$ and $\tau_g(q_i) \to 0$, we can select $i$ sufficiently large so that 
	\[
	\hat{d}_{\tau_g}(q_0,q_i) < \varepsilon/2 \quad \text{ and } \quad \tau_g(q_i) < \tilde{t}.
	\]
	Then
	\begin{align*}
		\dist(q_0, M_t) &\leq \hat{d}_{\tau_g}(q_0, q) \\
		&\leq \hat{d}_{\tau_g}(q_0, q_i) + \hat{d}_{\tau_g}(q_i, \gamma_q(\tau_g(q_i))) + \hat{d}_{\tau_g}(\gamma_q(\tau_g(q_i)), q) \\
		&< \varepsilon / 2 + \diam_{\hat{d}_{\tau_g}}(M_{\tau_g(q_i)}) + (t - \tau_g(q_i)) \\
		&< t + \varepsilon.
	\end{align*}
	Taking the limit $\varepsilon \to 0$, we conclude that
	\begin{equation}
		\label{eqn:t_estimate_2}
		\dist(q_0, M_t) \leq t.
	\end{equation}
	Combining estimates~\eqref{eqn:t_estimate_1} and~\eqref{eqn:t_estimate_2}, we obtain
	\[
		d_{H}( M_t, \overline{\tau_g}^{-1}(0) ) \leq t,
	\]
	which implies that $M_t$ converges to $\overline{\tau_g}^{-1}(0)$ in the Hausdorff sense as $t \to 0$.
\end{proof}

\clearpage
\section{Appendix}

\subsection{Deferred Proofs}	
\label{chap:Deferred_Proofs}
	
\begin{lemma*}[Lemma~\ref{lem:causal_timelike} (restated)]
	Let $X,T \in TM$ be tangent vectors at the same point, with $X$ causal and $T$ timelike. Then $g(X,T) \neq 0$.
\end{lemma*}
\begin{proof}
	Suppose, for contradiction, that $g(X,T) = 0$. Since $T$ is timelike and the metric $g$ has signature $(1,n)$, $X$ cannot be timelike and must therefore be null. Given that $g(X,T) = 0$ and $X$ is null, equality is attained in the reverse Cauchy--Schwarz inequality, which implies that $X$ and $T$ are linearly dependent. This is a contradiction, as $X$ is null and $T$ is timelike.
\end{proof}

\begin{lemma*}[Lemma~\ref{lem:past_timelike} (restated)]
	Let $p \in M$, and let $h_p$ be any Euclidean inner product on $T_pM$. Then the following conditions on a tangent vector $T \in T_pM$ are equivalent:
	\begin{enumerate}
		\item \label{item:1_past_timelike} $T$ is past-directed timelike.
		\item \label{item:2_past_timelike} For every future-directed causal vector $X \in T_pM$, $g(X,T) > 0$.
		\item \label{item:3_past_timelike} There exists a constant $C > 0$ such that, for all future-directed causal vectors $X \in T_pM$,
		\[
			g(T,X) \geq C\norm{X}_{h_p}.
		\]
	\end{enumerate}
\end{lemma*}
\begin{proof}
	\cref{item:1_past_timelike}$\implies$\cref{item:2_past_timelike}. By Lemma~\ref{lem:causal_timelike}, for every future-directed causal vector $X \in T_pM$, we have $g(X,T) \neq 0$. Suppose, for contradiction, that $g(X,T) < 0$ for some such $X$. Let $\vartheta$ denote the time orientation, and define
	\[
		T' = T + \lambda\vartheta,
	\]
	where $\lambda < 0$ is chosen so that 
	\[
		g(X,T') = g(X,T) + \lambda g(X,\vartheta) = 0.
	\]
	Then $T'$ is timelike and $g(X,T') = 0$, contradicting Lemma~\ref{lem:causal_timelike}.
	
	\cref{item:2_past_timelike}$\implies$\cref{item:3_past_timelike}. Suppose, for contradiction, that the conclusion fails. Then there exists a sequence of future-directed causal vectors $X_j \in T_pM$ such that 
	\[
		g(T,X_j) < j^{-1}\norm{X_j}_{h_p}. 
	\]
	By rescaling and passing to a subsequence if necessary, we may assume that $\norm{X_j}_{h_p} = 1$ for all $j$ and that $X_j \to X_0$ for some future-directed causal vector $X_0 \in T_pM$. Passing to the limit in the inequality yields
	\[
		g(T,X_0) = 0,
	\]
	which contradicts~\cref{item:2_past_timelike}. 
	
	\cref{item:3_past_timelike}$\implies$\cref{item:1_past_timelike}. Suppose, for contradiction, that $T$ is either zero, null, or spacelike. 
	Then there exists a future-directed causal vector $X \in T_pM$ such that $g(T,X) = 0$, contradicting the inequality in~\cref{item:3_past_timelike}.  Hence, $T$ must be timelike. Moreover, since the time orientation $\vartheta$ is future-directed timelike, the inequality
	\[
		g(T,\vartheta) \geq C\norm{\vartheta}_h > 0
	\]
	implies that $T$ is past-directed.
\end{proof}

\begin{lemma*}[Lemma~\ref{lem:time_functions_gradient} (restated)]
	Let $\tau$ be a time function on $M$. If $\tau$ is differentiable at a point $p \in M$, then the gradient vector $(\nabla \tau)_p$ is either past-directed causal or zero.
\end{lemma*}
\begin{proof}
	We first claim that for every future-directed timelike vector $T \in T_pM$, $g((\nabla \tau)_p, T) \geq 0$. Indeed, let $\gamma \from (-\delta,\delta) \to M$ be a future-directed timelike curve with $\gamma(0) = p$ and $\gamma'(0) = T$. Since $\tau$ is a time function, the composition $\tau\circ\gamma$ is strictly increasing. Hence, its derivative at zero satisfies $(\tau\circ\gamma)'(0) \geq 0$. This derivative equals $\mathrm{d}_p\tau(T) = g((\nabla \tau)_p, T)$, so the claim follows.
	
	Suppose, for contradiction, that $(\nabla \tau)_p$ is spacelike, that is, $g((\nabla \tau)_p, (\nabla \tau)_p) > 0$. Choose a future-directed timelike vector $T \in T_pM$ such that $g(T, (\nabla \tau)_p) = 0$, and define
	\[
		T' = T + \lambda(\nabla \tau)_p,
	\]
	where $\lambda < 0$ is chosen so that
	\[
		g(T,T) + \lambda^2g((\nabla \tau)_p, (\nabla \tau)_p) < 0.
	\]
	Let $\vartheta$ denote the time orientation. Then:
	\begin{enumerate}
		\item $T'$ is timelike, since
		\begin{align*}
			g(T',T') &= g(T + \lambda (\nabla \tau)_p, T + \lambda (\nabla \tau)_p) \\
			&= g(T,T) +2\lambda g(T, (\nabla \tau)_p) + \lambda^2 g((\nabla \tau)_p, (\nabla \tau)_p) \\
			&= g(T,T) + \lambda^2 g((\nabla \tau)_p, (\nabla \tau)_p) \\
			&< 0.
		\end{align*}
		\item $T'$ is future-directed, since
		\[
			g(T', \vartheta) = g(T, \vartheta) + \lambda g((\nabla \tau)_p, \vartheta) < 0,
		\]
		where we used that $g(T, \vartheta) < 0$ because $T$ is future-directed, and that $g((\nabla \tau)_p, \vartheta) \geq 0$ by the initial claim.
	\end{enumerate}
	However,
	\begin{align*}
		g((\nabla \tau)_p, T') &= g((\nabla \tau)_p, T) + \lambda g((\nabla \tau)_p, (\nabla \tau)_p) \\
		&= \lambda g((\nabla \tau)_p, (\nabla \tau)_p) \\
		&< 0,
	\end{align*}
	which contradicts the initial claim. Hence, $(\nabla \tau)_p$ cannot be spacelike and must be either causal or zero. Finally, if $(\nabla \tau)_p$ is causal, then it must be past-directed since $g((\nabla \tau)_p, \vartheta) \geq 0$.
\end{proof}

\clearpage
\printbibliography

\end{document}